\documentclass{article}

\usepackage[margin=1in]{geometry}
\usepackage{amsmath}
\usepackage{amssymb}
\usepackage{amsthm}
\usepackage{multirow}
\usepackage{enumerate}
\usepackage{color}
\usepackage{longtable}
\usepackage{caption2}

\usepackage{float}

\floatstyle{plaintop}
\restylefloat{table}

\allowdisplaybreaks[1]
\newcommand{\ra}[1]{\renewcommand{\arraystretch}{#1}}
\newcommand{\Mod}[1]{\ (\mathrm{mod}\ #1)}

\newtheorem{theorem}{Theorem}

\newtheorem{lemma}[theorem]{Lemma}
\newtheorem{proposition}[theorem]{Proposition}
\newtheorem{definition}[theorem]{Definition}
\newtheorem{example}[theorem]{Example}

\newtheorem{conjecture}[theorem]{Conjecture}

\numberwithin{theorem}{section}
\numberwithin{equation}{section}
\numberwithin{table}{section}

\newtheorem{remark}[theorem]{Remark}

\newcommand{\cB}{\mathcal{B}}

\newcommand{\cV}{\mathcal{V}}

\newcommand{\al}{\alpha}

\newcommand{\be}{\beta}

\newcommand{\pr}{\prime}

\newcommand{\sm}{\setminus}

\newcommand{\lan}{\langle}
\newcommand{\ran}{\rangle}

\newcommand{\lc}{\lceil}
\newcommand{\rc}{\rceil}

\newcommand{\Tr}{\text{Tr}}

\newcommand{\F}{\mathbb{F}}
\newcommand{\Fp}{\mathbb{F}_p}
\newcommand{\Fpm}{\mathbb{F}_{p^m}}
\newcommand{\Fq}{\mathbb{F}_q}

\newcommand{\STS}{\text{STS}}

\long\def\symbolfootnote[#1]#2{\begingroup%
\def\thefootnote{\fnsymbol{footnote}}\footnote[#1]{#2}\endgroup}

\begin{document}

\title{On the intersection distribution of degree three polynomials and related topics}
\author{Gohar Kyureghyan \and Shuxing Li \and Alexander Pott}
\date{}
\maketitle

\symbolfootnote[0]{
G.~Kyureghyan is with Institute of Mathematics, University of Rostock, 18057 Rostock, Germany (email: gohar.kyureghyan@uni-rostock.de). S.~Li and A.~Pott are with Institute of Algebra and Geometry, Faculty of Mathematics, Otto von Guericke University Magdeburg, 39106 Magdeburg, Germany (e-mail: shuxing\_li@sfu.ca, alexander.pott@ovgu.de).
}

\begin{abstract}
The intersection distribution of a polynomial $f$ over finite field $\Fq$ was recently proposed in Li and Pott (arXiv:2003.06678v1), which concerns the collective behaviour of a collection of polynomials $\{f(x)+cx \mid c \in \Fq\}$. The intersection distribution has an underlying geometric interpretation, which indicates the intersection pattern between the graph of $f$ and the lines in the affine plane $AG(2,q)$. When $q$ is even, the long-standing open problem of classifying o-polynomials can be rephrased in a simple way, namely, classifying all polynomials which have the same intersection distribution as $x^2$. Inspired by this connection, we proceed to consider the next simplest case and derive the intersection distribution for all degree three polynomials over $\Fq$ with $q$ both odd and even. Moreover, we initiate to classify all monomials having the same intersection distribution as $x^3$, where some characterizations of such monomials are obtained and a conjecture is proposed. In addition, two applications of the intersection distributions of degree three polynomials are presented. The first one is the construction of nonisomorphic Steiner triple systems and the second one produces infinite families of Kakeya sets in affine planes with previously unknown sizes.


\end{abstract}

\section{Introduction}\label{sec1}

Throughout this paper, let $\Fq=\Fpm$ be a finite field with characteristic $p$ and $f$ a polynomial over $\Fq$. The intersection distribution of $f$ originates from an elementary problem concerning the interaction between the graph $\{ (x,f(x)) \mid x \in \Fq \}$ of $f$ and the lines in the classical affine plane $AG(2,q)$. More precisely, for $0 \le i \le q$, we ask about the number of affine lines intersecting the graph of $f$ in exactly $i$ points. Note that there are $q$ vertical affine lines of the form $\{(x,y) \mid y \in \Fq\}$, where $x$ ranges over $\Fq$. Since each of these vertical lines intersects $G_f$ in exactly one point, we shall omit them and restrict to the remaining $q^2$ non-vertical lines. As an attempt to answer this question, the following concept of intersection distribution was proposed in \cite{LP}.

\begin{definition}[Intersection distribution]\label{def-intdis}
For $0 \le i \le q$, define
$$
v_i(f)=|\{(b,c) \in \Fq^2 \mid \mbox{$f(x)-bx-c=0$ has $i$ solutions in $\Fq$} \}|.
$$
The sequence $(v_i(f))_{i=0}^{q}$ is the intersection distribution of $f$. The integer $v_0(f)$ is the non-hitting index of $f$.
\end{definition}

We remark that for $0 \le i \le q$, there are exactly $v_i(f)$ non-vertical lines, which intersect the graph $\{(x,f(x)) \mid x \in \Fq\}$ in exactly $i$ points. In particular, the non-hitting index $v_0(f)$ is the number of affine lines which does not hit the graph of $f$. The intersection distribution of a polynomial $f$ conveys considerable information of $f$. For instance, the non-hitting index $v_0(f)$ measures the distance from $f$ to linear functions, and to the so called o-polynomial (when $q$ is even) or to $x^2$ (when $q$ is odd) \cite[Result 1.7]{LP}. Thus, intersection distribution serves as a new viewpoint to distinguish polynomials, which is different from the usual extended-affine equivalence ({\rm\cite[p. 1142]{BCP}}) and the Carlet-Charpin-Zinoviev equivalence {\cite[Definition 1]{BCP}, \rm\cite[Proposition 3]{CCZ}}. Moreover, the aforementioned geometric interpretation indicates that for the point set in the classical projective plane $PG(2,q)$ arising from a polynomial $f$, detailed information follows from the intersection distribution $f$ (see for instance \cite[Proposition 3.2]{LP}).

Having explained the reason why the intersection distribution is interesting, we proceed to consider its computation. First, we have the following basic equations, which essentially have been stated in \cite[Proposition 2.1]{LP} (see also \cite[Lemma 12.1]{Hirs}).

\begin{proposition}\label{prop-basiceqn}
Let $f$ be a polynomial over $\Fq$. The following equations hold.
\begin{align*}
\sum_{i=0}^{q} v_i(f)&=q^2, \\
\sum_{i=1}^{q} iv_i(f)&=q^2, \\
\sum_{i=2}^{q} i(i-1)v_i(f)&=q(q-1).
\end{align*}
\end{proposition}

Secondly, to facilitate the computation of the intersection distribution, the following definition was proposed in \cite[Definition 1.1(2)]{LP}.

\begin{definition}[Multiplicity distribution]\label{def-muldis}
Let $f$ be a polynomial over $\Fq$. For $b \in \Fq$ and $0 \le i \le q$, define
$$
M_i(f,b)=|\{ c \in \Fq \mid \mbox{$f(x)-bx-c=0$ has $i$ solutions in $\Fq$} \}|.
$$
The sequence $(M_i(f,b))_{i=0}^q$ is the multiplicity distribution of $f$ at $b$. The multiset of sequences $\{(M_i(f,b))_{i=0}^q \mid b \in \Fq\}$ is the multiplicity distribution of $f$.
\end{definition}
By definition, for $0 \le i \le q$, there are exactly $M_i(f,b)$ lines among the parallel class of $q$ affine lines $\{ y=bx+c \mid c \in \Fq \}$, which intersect the graph of $f$ in $i$ points. From now on, we use $\Fq^*$ to denote the set of nonzero elements in $\Fq$.

\begin{remark}\label{rem-normalization}
\quad
\begin{itemize}
\item[(1)] By definition, for a polynomial $f$ over $\Fq$ and $0 \le i \le q$, we have
$$
v_i(f)=\sum_{b \in \Fq}M_i(f,b).
$$
Hence, the multiplicity distribution of $f$ implies its intersection distribution.
\item[(2)] Let $f(x)=\sum_{i=0}^n a_ix^i$, where $n \ge 2$ and $a_n \ne 0$. Note that for each $0 \le i \le q$, we have  $M_i(f,b)=M_i(a_n^{-1}(f-a_1x-a_0),a_n^{-1}(b-a_1))$. Hence, in order to compute the intersection distribution of $f$, one can assume without generality that $a_1=a_0=0$ and $a_n=1$.
\item[(3)] Let $f$ be a permutation polynomial and $f^{-1}$ be its inverse. Clearly, $M_1(f,0)=M_1(f^{-1},0)=q$. Moreover, note that for $b \in \Fq^*$, the two equations $f(x)-bx-c=0$ and $f^{-1}(x)-\frac{1}{b}x+\frac{c}{b}=0$ have the same number of solutions. Hence, $f$ and $f^{-1}$ have the same multiplicity distribution and therefore, the same intersection distribution.
\end{itemize}
\end{remark}

We remark that in general, computing the intersection and multiplicity distributions is a nontrivial problem. In \cite[Appendix B]{LP}, the multiplicity distributions of monomials $x^d$ over $\Fq$, where $d \in \{p^i,p^i+1,\frac{q-1}{2},\frac{q+1}{2},q-2,q-1\}$, have been determined. Indeed, combining \cite[Propositions B.1, B.9]{LP} and Remark~\ref{rem-normalization}(1)(2), we have the following proposition. For the sake of simplicity, from now on, we only list the first few values of the intersection distribution $v_i(f)$ and multiplicity distribution $M_i(f,b)$ with $i$ at most $4$, and the unmentioned values are all zeros.

\begin{proposition}
Let $x^2$ be a monomial over $\Fq$.
\begin{itemize}
\item[(1)] If $p=2$, then
$$
\begin{cases}
  M_0(x^2,0)=0, \quad M_1(x^2,0)=q, \quad M_2(x^2,0)=0, \\
  M_0(x^2,b)=\frac{q}{2}, \quad M_1(x^2,b)=0, \quad M_2(x^2,b)=\frac{q}{2}, & \mbox{if $b \ne 0$.}
\end{cases}
$$
\item[(2)] If $p$ is odd, then for each $b \in \Fq$,
$$
  M_0(x^2,b)=\frac{q-1}{2}, \quad M_1(x^2,b)=1, \quad M_2(x^2,b)=\frac{q-1}{2}.
$$
\end{itemize}
In particular, for each polynomial $f$ over $\Fq$ with degree two, we have
\begin{equation}\label{eqn-intdisdegreetwo}
v_0(f)=\frac{q(q-1)}{2}, \quad v_1(f)=q, \quad v_2(f)=\frac{q(q-1)}{2}.
\end{equation}
\end{proposition}

Consequently, the intersection distribution of polynomials with degree two is clear. We remark that $f$ having degree two forces $v_i(f)=0$ for each $i>2$, so that the intersection distribution \eqref{eqn-intdisdegreetwo} follows from Proposition~\ref{prop-basiceqn}. A natural question is, if we drop the degree two condition, is there any other polynomial which has intersection distribution \eqref{eqn-intdisdegreetwo}? Historically, this problem has been intensively studied in terms of classifying ovals or hyperovals in the classical projective planes (see \cite[Chapter 8]{Hirs} for instance). When $q$ is odd, a famous result due to Segre \cite{Seg} indicates each polynomial $f$ satisfying \eqref{eqn-intdisdegreetwo} is in some sense equivalent to $x^2$. On the other hand, when $q$ is even, the situation is much more subtle. In this case, a polynomial $f$ with the same intersection distribution as $x^2$ is called an \emph{o-polynomial}. The classification of o-polynomials, especially o-monomials, is a long-standing problem which has attracted much attention (see \cite{CS,Che,Hirs,Vis} and the references therein).


In this paper, we pursue a result analogous to the one stated above. More precisely, we take one step forward to consider the intersection distribution of degree three polynomials. This is the next simplest case as the degree three condition ensures that $v_i(f)=0$ for each $i>3$. Together with Proposition~\ref{prop-basiceqn}, the intersection distribution of each degree three polynomial $f$ can be determined by exactly one of $v_i(f)$, $0 \le i \le 3$. By Remark~\ref{rem-normalization}(2), it suffices to determine the intersection distribution of $x^3-ax^2$ for each $a \in \Fq$, and we have the following complete description.

\begin{theorem}\label{thm-cubicintdis}
Let $f(x)=x^3-ax^2$ be a polynomial over $\Fq$. If $p \ne 3$, then we have
\begin{equation}\label{eqn-intdisdeg3othchar}
v_0(f)=\frac{q^2-1}{3}, \quad v_1(f)=\frac{q^2-q+2}{2}, \quad v_2(f)=q-1, \quad v_3(f)=\frac{q^2-3q+2}{6}.
\end{equation}
If $p=3$ and $a=0$, then we have
\begin{equation}\label{eqn-intdisdeg3char3}
v_0(f)=\frac{q(q-1)}{3}, \quad v_1(f)=\frac{q(q+1)}{2}, \quad v_2(f)=0, \quad v_3(f)=\frac{q(q-1)}{6}.
\end{equation}
If $p=3$ and $a \ne 0$, then we have
$$
v_0(f)=\frac{q^2}{3}, \quad v_1(f)=\frac{q(q-1)}{2}, \quad v_2(f)=q, \quad v_3(f)=\frac{q(q-3)}{6}.
$$
\end{theorem}

In order to derive the above theorem, we present a detailed computation determining the multiplicity distribution of degree three polynomial $x^3-ax^2$ in Section ~\ref{sec2}. To achieve this, we consider the number of $\Fq$-solutions to
$$
x^3-\be x-c=0, \quad \mbox{if $p \ne 3$,} \quad \mbox{or} \quad x^3-x^2-c=0, \quad \mbox{if $p=3$,}
$$
where $c \in \Fq$ and $\be$ is either $1$ or a primitive element of $\Fq$. When the equation has at least one solution $x_0 \in \Fq$, we give a characterization of the number of $\Fq$-solutions in terms of $x_0$. In Section~\ref{sec3}, we proceed to consider a much more challenging problem, namely, determining all monomials which has the same intersection distribution as $x^3$. Although a complete answer is by far elusive, we make some detailed analysis and present strong restrictions to these monomials. Moreover, based on the numerical experiment, we propose a conjecture classifying all monomials having the same intersection distribution as $x^3$. As an application, in Section~\ref{sec4}, we observe that polynomials over $\F_{3^m}$ with intersection distribution \eqref{eqn-intdisdeg3char3} produces Steiner triples systems. Interestingly, some numerical results indicate that certain distinct monomials satisfying \eqref{eqn-intdisdeg3char3} generate nonisomorphic Steiner triple systems. In Section~\ref{sec5}, applying the multiplicity distribution of $x^3-ax^2$, we construct several infinite families of Kakeya sets in affine planes, whose sizes are different from the known ones. Section~\ref{sec6} concludes the paper.

\section{The multiplicity and intersection distributions of degree three polynomials}\label{sec2}

In this section, we consider the multiplicity distribution of degree three polynomial. In view of Remark~\ref{rem-normalization}(2), we only need to consider the degree three polynomial of the form $x^3-ax^2$, where $a \in \Fq$.

From now on, we always denote a primitive element of finite field $\Fq$ by $\al$. Given a finite field $\Fq$ and a positive integer $N \mid q-1$, we use $C_0^{(N,q)}$ to denote the set consisting of nonzero $N$-th powers in $\Fq$. Suppose $\al$ is a primitive element of $\Fq$, then for $0 \le i \le N-1$, define $C_i^{(N,q)}=\al^iC_0^{(N,q)}=\{\al^ix \mid x \in C_0^{(N,q)}\}$. Hence, when $q$ is odd, we know that $C_0^{(2,q)}$ is the set of nonzero squares and $C_1^{(2,q)}$ is the set of nonsquares in $\Fq$.

The following proposition says, roughly speaking, in order to determine the multiplicity distribution of $x^3-ax^2$, it suffices to compute $M_i(x^3,0)$, $M_i(x^3,1)$, $M_i(x^3,\al)$ and when $p=3$, also $M_i(x^3-x^2,0)$.

\begin{lemma}\label{lem-cubicreduction}
\quad
\begin{itemize}
\item[(1)] When $p=2$, we have
$$
M_i(x^3,b)=\begin{cases}
                  M_i(x^3,0), & \mbox{if $b=0$,} \\
                  M_i(x^3,1), & \mbox{if $b \ne 0$,} \\
\end{cases}
$$
and
$$
M_i(x^3-ax^2,b)=M_i(x^3,\frac{b}{a^2}+1)=\begin{cases}
                  M_i(x^3,0), & \mbox{if $\frac{b}{a^2}=1$,} \\
                  M_i(x^3,1), & \mbox{if $\frac{b}{a^2} \ne 1$,} \\
\end{cases}
$$
where $a \ne 0$ and $b \in \Fq$.
\item[(2)] When $p=3$, we have
$$
M_i(x^3,b)=\begin{cases}
                  M_i(x^3,0), & \mbox{if $b=0$,} \\
                  M_i(x^3,1), & \mbox{if $b \in C_0^{(2,q)}$,} \\
                  M_i(x^3,\al), & \mbox{if $b \in C_1^{(2,q)}$.} \\
\end{cases}
$$
For $a \ne 0$ and $b \in \Fq$, we have $M_i(x^3-ax^2,b)=M_i(x^3-x^2,0)$.
\item[(3)] When $p>3$, we have
$$
M_i(x^3,b)=\begin{cases}
                  M_i(x^3,0), & \mbox{if $b=0$,} \\
                  M_i(x^3,1), & \mbox{if $b \in C_0^{(2,q)}$,} \\
                  M_i(x^3,\al), & \mbox{if $b \in C_1^{(2,q)}$.} \\
\end{cases}
$$
and
$$
M_i(x^3-ax^2,b)=M_i(x^3,\frac{b}{a^2}+\frac{1}{3})=\begin{cases}
                  M_i(x^3,0), & \mbox{if $\frac{b}{a^2}=-\frac{1}{3}$,} \\
                  M_i(x^3,1), & \mbox{if $\frac{b}{a^2}+\frac{1}{3} \in C_0^{(2,q)}$,} \\
                  M_i(x^3,\al), & \mbox{if $\frac{b}{a^2}+\frac{1}{3} \in C_1^{(2,q)}$,}
\end{cases}
$$
where $a \ne 0$ and $b \in \Fq$.
\end{itemize}
\end{lemma}
\begin{proof}
In all the three cases, the expression of $M_i(x^3,b)$ is clear. So we only need to consider $M_i(x^3-ax^2,b)$ with $a \ne 0$. For this purpose, we consider the number of solutions in $\Fq$ to the equation $x^3-ax^2-bx-c=0$.

If $p=2$ or $p>3$, namely, $\gcd(3,q)=1$, dividing $a^3$ on both sides and replacing $\frac{x}{a}$ with $x$, we have $x^3-x^2-\frac{b}{a^2}x-\frac{c}{a^3}=0$. Replacing $x$ with $x+\frac{1}{3}$ in the latter equation leads to $x^3-(\frac{b}{a^2}+\frac{1}{3})x-(\frac{c}{a^3}+\frac{b}{3a^2}+\frac{2}{27})=0$. Note that fixing $a$ and $b$, when $c$ ranges over $\Fq$, so does $\frac{c}{a^3}+\frac{b}{3a^2}+\frac{2}{27}$. Hence, $M_i(x^3-ax^2,b)=M_i(x^3,\frac{b}{a^2}+\frac{1}{3})$. Note that when $p=2$, $M_i(x^3,\frac{b}{a^2}+1)=M_i(x^3,\frac{b}{a^2}+\frac{1}{3})$. The rest follows easily.

If $p=3$, namely, $\gcd(3,q)=3$, dividing $a^3$ on both sides and replacing $\frac{x}{a}$ with $x$, we have $x^3-x^2-\frac{b}{a^2}x-\frac{c}{a^3}=0$. Replacing $x$ with $x-\frac{b}{2a^2}$, we have $x^3-x^2-(\frac{c}{a^3}-\frac{b^2}{4a^4}+\frac{b^3}{8a^6})=0$. Note that fixing $a$ and $b$, when $c$ ranges over $\Fq$, so does $\frac{c}{a^3}-\frac{b^2}{4a^4}+\frac{b^3}{8a^6}$. Hence, $M_i(x^3-ax^2,b)=M_i(x^3-x^2,0)$.
\end{proof}

Note that $M_i(x^3,0)$ is easy to compute. Moreover, when $p=3$, $M_i(x^3,1)$ and $M_i(x^3,\al)$ are also straightforward.

\begin{lemma}\label{lem-cubiczero}
\begin{itemize}
\item[(1)] When $p=2$, we have
$$
\begin{cases}
M_0(x^3,0)=0, \quad M_1(x^3,0)=q, \quad M_2(x^3,0)=0, \quad M_3(x^3,0)=0, & \mbox{if $m$ odd,} \\
M_0(x^3,0)=\frac{2(q-1)}{3}, \quad M_1(x^3,0)=1, \quad M_2(x^3,0)=0, \quad M_3(x^3,0)=\frac{q-1}{3}, & \mbox{if $m$ even.}
\end{cases}
$$
\item[(2)] When $p=3$, we have
\begin{align*}
&M_0(x^3,0)=0, \quad M_1(x^3,0)=q, \quad M_2(x^3,0)=0, \quad M_3(x^3,0)=0, \\
&M_0(x^3,1)=\frac{2q}{3}, \quad M_1(x^3,1)=0, \quad M_2(x^3,1)=0, \quad M_3(x^3,1)=\frac{q}{3}, \\
&M_0(x^3,\al)=0, \quad M_1(x^3,\al)=q, \quad M_2(x^3,\al)=0, \quad M_3(x^3,\al)=0.
\end{align*}
\item[(3)] When $p>3$, we have
$$
\begin{cases}
M_0(x^3,0)=\frac{2(q-1)}{3}, \quad M_1(x^3,0)=1, \quad M_2(x^3,0)=0, \quad M_3(x^3,0)=\frac{q-1}{3}, & \mbox{if $q \equiv 1 \Mod{3}$,} \\
M_0(x^3,0)=0, \quad M_1(x^3,0)=q, \quad M_2(x^3,0)=0, \quad M_3(x^3,0)=0, & \mbox{if $q \equiv 2 \Mod{3}$.}
\end{cases}
$$
\end{itemize}
\end{lemma}

Now we introduce the concept of cyclotomic number, which will be used later. For $0 \le i,j \le 1$, define the cyclotomic numbers of order $2$ as
$$
(i,j)_q=|(1+C_i^{(2,q)}) \cap C_j^{(2,q)}|.
$$
The cyclotomic numbers of order $2$ are well known, see for instance \cite{Sto}.

\begin{lemma}\label{lem-cyclotomy}
Let $q$ be an odd prime power. If $q \equiv 1 \Mod{4}$, we have
$$
(0,0)_q=\frac{q-5}{4}, \quad (0,1)_q=(1,0)_q=(1,1)_q=\frac{q-1}{4}.
$$
If $q \equiv 3 \Mod{4}$, we have
$$
(0,1)_q=\frac{q+1}{4}, \quad (0,0)_q=(1,0)_q=(1,1)_q=\frac{q-3}{4}.
$$
\end{lemma}

Employing the cyclotomic numbers of order $2$, we proceed to prove the following preparatory lemma.

\begin{lemma}\label{lem-cubicspecial}
Let {\rm$\Tr$} be the absolute trace defined on $\Fq$. For the equations $x^3-x-c=0$ or $x^3-\al x-c=0$, assume that $x_0 \in \Fq$ is a solution.
\begin{itemize}
\item[(1)] When $p=2$, we have
$$
|\{x \in \Fq \mid x^3-x-c=0\}|=\begin{cases}
                                   1 & \mbox{if {\rm $\Tr(\frac{1}{x_0})=\Tr(\frac{1}{c})=\Tr(1)+1$,}} \\
                                   2 & \mbox{if $x_0 \in \{0,1\}$, or equivalently, $c=0$,} \\
                                   3 & \mbox{if {\rm $\Tr(\frac{1}{x_0})=\Tr(\frac{1}{c})=\Tr(1)$, $x_0 \ne 1$.}}
                               \end{cases}
$$
Consequently,
$$
\begin{cases}
  M_0(x^3,1)=\frac{q+1}{3}, \quad M_1(x^3,1)=\frac{q}{2}-1, \quad M_2(x^3,1)=1, \quad M_3(x^3,1)=\frac{q-2}{6}, &\mbox{if $m$ odd,}\\
  M_0(x^3,1)=\frac{q-1}{3}, \quad M_1(x^3,1)=\frac{q}{2}, \quad M_2(x^3,1)=1, \quad M_3(x^3,1)=\frac{q-4}{6}, &\mbox{if $m$ even.}
\end{cases}
$$
\item[(2)] When $p=3$, we have
$$
|\{x \in \Fq \mid x^3-x^2-c=0\}|=\begin{cases}
                                   1 & \mbox{if $2x_0+1 \in C_{1}^{(2,q)}$,} \\
                                   2 & \mbox{if $x_0 \in \{0, 1\}$, or equivalently, $c=0$,} \\
                                   3 & \mbox{if $2x_0+1 \in C_{0}^{(2,q)}$ and $x_0 \ne 0$.}
                               \end{cases}
$$
Consequently, we have
$$
M_0(x^3-x^2,0)=\frac{q}{3}, \quad M_1(x^3-x^2,0)=\frac{q-1}{2}, \quad M_2(x^3-x^2,0)=1, \quad M_3(x^3-x^2,0)=\frac{q-3}{6}.
$$
\item[(3)] When $p>3$, for $\be \in \{ 1,\al \}$ and
$$
j=\begin{cases}
     1 & \mbox{if $\be=1$,}\\
    -1 & \mbox{if $\be=\al$,}
\end{cases}
$$
we have
$$
|\{x \in \Fq \mid x^3-\be x-c=0\}|=\begin{cases}
                                   1 & \mbox{if $1-\frac{3x_0^2}{4\be} \in \be C_{1}^{(2,q)}$,} \\
                                   2 & \mbox{if $3 \in \be C_0^{(2,q)}$ and $x_0^2 \in \{\frac{\be}{3}, \frac{4\be}{3}\}$,} \\
                                   3 & \mbox{if $1-\frac{3x_0^2}{4\be} \in \be C_{0}^{(2,q)}$ and $x_0^2 \ne \frac{\be}{3}$ whenever $3 \in \be C_0^{(2,q)}$.}
                               \end{cases}
$$
Consequently, we have
$$
\begin{cases}
  M_0(x^3,\be)=\frac{q-1}{3}, \; M_1(x^3,\be)=\frac{q-j}{2}, \; M_2(x^3,\be)=1+j, \; & M_3(x^3,\be)=\frac{q-4-3j}{6}, \\
                                                                                     & \mbox{if $p \equiv 1 \Mod{12}$ or $m$ even,}\\
  M_0(x^3,\be)=\frac{q-1}{3}, \; M_1(x^3,\be)=\frac{q+j}{2}, \; M_2(x^3,\be)=1-j, \; & M_3(x^3,\be)=\frac{q-4+3j}{6}, \\
                                                                                     & \mbox{if $p \equiv 7 \Mod{12}$ and $m$ odd,}\\
  M_0(x^3,\be)=\frac{q+1}{3}, \; M_1(x^3,\be)=\frac{q-2+j}{2}, \; M_2(x^3,\be)=1-j,\;& M_3(x^3,\be)=\frac{q-2+3j}{6}, \\
                                                                                     & \mbox{if $p \equiv 5 \Mod{12}$ and $m$ odd,}\\
  M_0(x^3,\be)=\frac{q+1}{3}, \; M_1(x^3,\be)=\frac{q-2-j}{2}, \; M_2(x^3,\be)=1+j,\;& M_3(x^3,\be)=\frac{q-2-3j}{6}, \\
                                                                                     & \mbox{if $p \equiv 11 \Mod{12}$ and $m$ odd.}\\
\end{cases}
$$
\end{itemize}
\end{lemma}
\begin{proof}
We first prove (1). Suppose $x^3-x-c=0$ has exactly two solutions in $\Fq$. Then $x^3-x-c=(x-x_1)(x-x_2)^2$ for some distinct $x_1, x_2 \in \Fq$. Comparing the coefficients, we have $x_1=0$, $x_2=1$, and $c=0$. Therefore, $M_2(x^3,1)=1$ as $x^3-x-c=0$ has exactly two solutions in $\Fq$ if and only if $c=0$ and the two solutions are $0$ and $1$. Next, we proceed to consider when $x^3-x-c=0$ has exactly one or three solutions. Suppose $x^3-x-c=0$ has one solution $x_0 \in \Fq$, then we can factor $x^3-x-c=(x-x_0)(x^2+x_0x+x_0^2-1)$, where $c=x_0^3+x_0$. We need to check if $x^2+x_0x+x_0^2-1=0$ has solutions in $\Fq$, where $x_0 \not\in \{0,1\}$. Note that $x^2+x_0x+x_0^2-1=0$ is equivalent to $1+\frac{1}{x_0^2}=(\frac{x}{x_0})^2+\frac{x}{x_0}$. Hence, $x^2+x_0x+x_0^2-1=0$ has no solution if and only if $\Tr(\frac{1}{x_0})=\Tr(1)+1$, and two solutions if and only if $\Tr(\frac{1}{x_0})=\Tr(1)$. Note that $\Tr(\frac{1}{c})=\Tr(\frac{1}{x_0+1}(\frac{1}{x_0}+\frac{1}{x_0+1}))=\Tr(\frac{1}{x_0}+\frac{1}{x_0+1}+\frac{1}{(x_0+1)^2})=\Tr(\frac{1}{x_0})$.
If $m$ is odd, then there are $\frac{q}{2}-1$ choices of $x_0 \in \Fq \sm\{0,1\}$, such that $\Tr(\frac{1}{x_0})=\Tr(1)+1=0$. Thus, $M_1(x^3,1)=\frac{q}{2}-1$. Moreover, there are $\frac{q}{2}-1$ choices of $x_0 \in \Fq \sm\{0,1\}$, such that $\Tr(\frac{1}{x_0})=\Tr(1)=1$. Thus, $M_3(x^3,1)=\frac{1}{3}(\frac{q}{2}-1)=\frac{q-2}{6}$ and therefore, the value of $M_0(x^3,1)$ follows immediately. Similar arguments lead to the values of $M_i(x^3,1)$ when $m$ is even.

The proofs of (2) and (3) are very similar to each other. Below, we only prove (3) with $\be=\al$. Suppose $x^3-\al x-c=0$ has exactly two solutions in $\Fq$. Then $x^3-\al x-c=(x-x_1)(x-x_2)^2$ for some distinct $x_1, x_2 \in \Fq$. Comparing the coefficients, we have $x_1^2=\frac{4\al}{3}$, $x_2^2=\frac{\al}{3}$, and $x_1+2x_2=0$. Hence, $x^3-\al x-c=0$ has two solutions in $\Fq$ if and only if $3 \in \al C_{0}^{(2,q)}$ and $c=\pm \frac{2\al}{3}\sqrt{\frac{\al}{3}}$. In this case, we have $M_2(x^3,\al)=2$, where $\{2\sqrt{\frac{\al}{3}}, -\sqrt{\frac{\al}{3}}\}$ and $\{-2\sqrt{\frac{\al}{3}}, \sqrt{\frac{\al}{3}}\}$ are two sets of solutions. Next, we proceed to consider when $x^3-\al x-c=0$ has exactly one or three solutions. Since $x^3-\al x-c=0$ has one solution $x_0 \in \Fq$, we can factor $x^3-\al x-c=(x-x_0)(x^2+x_0x+x_0^2-\al)$, where $c=x_0^3-\al x_0$. We need to check if $x^2+x_0x+x_0^2-\al=0$ has solutions in $\Fq$, where $x_0^2 \not\in \{\frac{\al}{3},\frac{4\al}{3}\}$. Since $x^2+x_0x+x_0^2-\al=0$ is equivalent to $(x+\frac{x_0}{2})^2=\al-\frac{3x_0^2}{4}$, it has zero or two solutions if and only if $1-\frac{3x_0^2}{4\al} \in C_0^{(2,q)}$ or $1-\frac{3x_0^2}{4\al} \in C_1^{(2,q)}$. We first consider the case of $q \equiv 1 \Mod{4}$. If $3 \in C_0^{(2,q)}$, then the number of nonzero square $x_0^2 \in C_0^{(2,q)}$, such that $1-\frac{3x_0^2}{4\al} \in C_0^{(2,q)}$, is equal to $(1,0)_q=\frac{q-1}{4}$. Note that when $x_0=0$, we have $1-\frac{3x_0^2}{4\al}=1 \in C_0^{(2,q)}$. Thus, $M_1(x^3,\al)=2\cdot\frac{q-1}{4}+1=\frac{q+1}{2}$. Similarly, the number of nonzero square $x_0^2 \in C_0^{(2,q)}$, such that $1-\frac{3x_0^2}{4\al} \in C_1^{(2,q)}$, is equal to $(1,1)_q=\frac{q-1}{4}$. Hence, $M_3(x^3,\al)=2\cdot\frac{q-1}{4}\cdot\frac{1}{3}=\frac{q-1}{6}$. If $3 \in C_1^{(2,q)}$, then the number of nonzero square $x_0^2 \in C_0^{(2,q)} \sm \{ \frac{\al}{3}, \frac{4\al}{3}\}$, such that $1-\frac{3x_0^2}{4\al} \in C_0^{(2,q)}$, is equal to $(0,0)_q=\frac{q-5}{4}$. Note that when $x_0=0$, we have $1-\frac{3x_0^2}{4\al}=1 \in C_0^{(2,q)}$. Thus, $M_1(x^3,\al)=2\cdot\frac{q-5}{4}+1=\frac{q-3}{2}$. Similarly, the number of nonzero square $x_0^2 \in C_0^{(2,q)} \sm \{ \frac{\al}{3}, \frac{4\al}{3}\}$, such that $1-\frac{3x_0^2}{4\al} \in C_1^{(2,q)}$, is equal to $(0,1)_q-1=\frac{q-5}{4}$. Note that a minus one appears in the previous equation, as $1-\frac{3x_0^2}{4\al} \in C_1^{(2,q)}$ when $x_0^2=\frac{\al}{3}$. Therefore, $M_3(x^3,\al)=2\cdot\frac{q-5}{4}\cdot\frac{1}{3}=\frac{q-5}{6}$. Note that $3 \in C_{0}^{(2,q)}$ if and only if $3 \in C_0^{(2,p)}$ or $m$ is even. Consequently, we have
$$
3 \in \begin{cases}
  C_0^{(2,q)} & \mbox{if $p \equiv 1,11 \Mod{12}$ or $m$ even,} \\
  C_1^{(2,q)} & \mbox{if $p \equiv 5,7 \Mod{12}$ and $m$ odd.}
\end{cases}
$$
Hence, $q \equiv 1 \Mod{4}$ and $3 \in C_0^{(2,q)}$ is equivalent to $p \equiv 1 \Mod{12}$ or $m$ even. Similarly, $q \equiv 1 \Mod{4}$ and $3 \in C_1^{(2,q)}$ is equivalent to $p \equiv 5 \Mod{12}$ and $m$ odd. Therefore, two out of the four cases in (3) with $\be=\al$ have been completed by the above argument. Applying an analogous approach to the $q \equiv 3 \Mod{4}$ case, we complete the proof of (3) with $\be=\al$.
\end{proof}

\begin{remark}
We note that for $\Fq=\F_{2^m}$, the value of $M_i(x^3,1)$ has been computed in {\rm\cite[Appendix]{KHCH}}. Moreover, as a special case, the multiplicity distribution of $x^{3}$ follows from the result of Bluher {\rm\cite[Theorem 5.6]{Blu}}, see also {\rm\cite[Proposition B.9]{LP}}.
\end{remark}

Combining Lemmas~\ref{lem-cubicreduction}, \ref{lem-cubiczero} and \ref{lem-cubicspecial}, we can completely determine the multiplicity distribution of degree three polynomials.

\begin{theorem}\label{thm-cubicmuldis}
The multiplicity distribution of $f(x)=x^3-ax^2$ over $\Fq=\F_{p^m}$ is as follows.
\begin{itemize}
\item[(1)] When $p=2$, if $m$ odd, then we have
$$
\begin{cases}
M_0(f,b)=0, M_1(f,b)=q, M_2(f,b)=0, M_3(f,b)=0, & \mbox{if $a=b=0$, or $a \ne 0$, $\frac{b}{a^2}=1$,} \\
M_0(f,b)=\frac{q+1}{3}, M_1(f,b)=\frac{q}{2}-1, M_2(f,b)=1, M_3(f,b)=\frac{q-2}{6}, & \mbox{if $a=0$, $b \ne 0$, or $a \ne 0$, $\frac{b}{a^2} \ne 1$,}
\end{cases}
$$
and if $m$ even, then we have
$$
\begin{cases}
M_0(f,b)=\frac{2(q-1)}{3}, M_1(f,b)=1, M_2(f,b)=0, M_3(f,b)=\frac{q-1}{3}, & \mbox{if $a=b=0$, or $a \ne 0$, $\frac{b}{a^2}=1$,} \\
M_0(f,b)=\frac{q-1}{3}, M_1(f,b)=\frac{q}{2}, M_2(f,b)=1, M_3(f,b)=\frac{q-4}{6}, & \mbox{if $a=0$, $b \ne 0$, or $a \ne 0$, $\frac{b}{a^2} \ne 1$.}
\end{cases}
$$
\item[(2)] When $p=3$, we have
$$
\begin{cases}
M_0(x^3,0)=0, \quad M_1(x^3,0)=q, \quad M_2(x^3,0)=0, \quad M_3(x^3,0)=0, & \mbox{if $b \in \{0\} \cup C_1^{(2,q)}$,} \\
M_0(x^3,b)=\frac{2q}{3}, \quad M_1(x^3,b)=0, \quad M_2(x^3,b)=0, \quad M_3(x^3,b)=\frac{q}{3}, & \mbox{if $b \in C_0^{(2,q)}$.}
\end{cases}
$$
Moreover,
$$
M_0(f,b)=\frac{q}{3}, \quad M_1(f,b)=\frac{q-1}{2}, \quad M_2(f,b)=1, \quad M_3(f,b)=\frac{q-3}{6}.
$$
where $a \ne 0$ and $b \in \Fq$.
\item[(3)] When $p>3$, if $p \equiv 1 \Mod{12}$ or $m$ even, or $p \equiv 7 \Mod{12}$ and $m$ odd, set
$$
j=\begin{cases}
     1 & \mbox{if $p \equiv 1 \Mod{12}$ or $m$ even,}\\
    -1 & \mbox{if $p \equiv 7 \Mod{12}$ and $m$ odd.}
\end{cases}
$$
Then we have
$$
\begin{cases}
M_0(f,b)=\frac{2(q-1)}{3}, M_1(f,b)=1, M_2(f,b)=0, M_3(f,b)=\frac{q-1}{3}, \\
\hspace{8cm} \mbox{if $a=b=0$, or $a \ne 0$, $\frac{b}{a^2}=-\frac{1}{3}$,} \\
M_0(f,b)=\frac{q-1}{3}, M_1(f,b)=\frac{q-j}{2}, M_2(f,b)=1+j, M_3(f,b)=\frac{q-4-3j}{6}, \\
\hspace{8cm} \mbox{if $a=0$, $b \in C_0^{(2,q)}$, or $a \ne 0$, $\frac{b}{a^2}+\frac{1}{3} \in C_0^{(2,q)}$,} \\
M_0(f,b)=\frac{q-1}{3}, M_1(f,b)=\frac{q+j}{2}, M_2(f,b)=1-j, M_3(f,b)=\frac{q-4+3j}{6}, \\
\hspace{8cm} \mbox{if $a=0$, $b \in C_1^{(2,q)}$, or $a \ne 0$, $\frac{b}{a^2}+\frac{1}{3} \in C_1^{(2,q)}$,}
\end{cases}
$$

If $p \equiv 5 \Mod{12}$ and $m$ odd, or $p \equiv 11 \Mod{12}$ and $m$ odd, set
$$
k=\begin{cases}
     1 & \mbox{if $p \equiv 5 \Mod{12}$ and $m$ odd,}\\
    -1 & \mbox{if $p \equiv 11 \Mod{12}$ and $m$ odd.}
\end{cases}
$$
Then we have
$$
\begin{cases}
M_0(f,b)=0, M_1(f,b)=q, M_2(f,b)=0, M_3(f,b)=0, \\
\hspace{8cm} \mbox{if $a=b=0$, or $a \ne 0$, $\frac{b}{a^2}=-\frac{1}{3}$,} \\
M_0(f,b)=\frac{q+1}{3}, M_1(f,b)=\frac{q-2+k}{2}, M_2(f,b)=1-k, M_3(f,b)=\frac{q-2+3k}{6}, \\
\hspace{8cm} \mbox{if $a=0$, $b \in C_0^{(2,q)}$, or $a \ne 0$, $\frac{b}{a^2}+\frac{1}{3} \in C_0^{(2,q)}$,} \\
M_0(f,b)=\frac{q+1}{3}, M_1(f,b)=\frac{q-2-k}{2}, M_2(f,b)=1+k, M_3(f,b)=\frac{q-2-3k}{6}, \\
\hspace{8cm} \mbox{if $a=0$, $b \in C_1^{(2,q)}$, or $a \ne 0$, $\frac{b}{a^2}+\frac{1}{3} \in C_1^{(2,q)}$,}
\end{cases}
$$
\end{itemize}
\end{theorem}

According to Remark~\ref{rem-normalization}(1), Theorem~\ref{thm-cubicintdis} immediately follows from Theorem~\ref{thm-cubicmuldis}.

\begin{remark}
Let $f \in \Fq[x]$ be a polynomial of degree $2 \le d \le q-1$. Define
$$
N_f=\{ c \in \Fq \mid \mbox{$f(x)+cx$ is a permutation over $\Fq$} \}.
$$
Lower and upper bounds on the non-hitting index $v_0(f)$ involving $q$, $d$ and $|N_f|$ were derived in {\rm\cite[Proposition 3.4]{LP}}. More precisely, we have
\begin{equation}\label{eqn-lowerupper}
\lc \frac{q-1}{d}\rc (q-|N_f|) \le v_0(f) \le (q-\lc \frac{q}{d} \rc)(q-|N_f|).
\end{equation}
Since the size of $N_f$ is in general difficult to compute, the tightness of the bounds in \eqref{eqn-lowerupper} remains unclear. On the other hand, Theorem~\ref{thm-cubicintdis} provides some instances where the bounds are actually tight, which can be achieved by polynomials of the form $x^3-ax^2$. In fact, the lower bound in \eqref{eqn-lowerupper} is tight, when $p=2$, $m$ odd, or $p \equiv 5,11 \Mod{12}$, $m$ odd, or $p=3$, $a \ne 0$. The upper bound in \eqref{eqn-lowerupper} is tight, when $p=3$ and $a=0$.
\end{remark}

Let $f$ be a polynomial over $\Fq$. For $b \in \Fq$ and $0 \le i \le q$, define
$$
M_i^*(f,b)=|\{ c \in \Fq \mid \mbox{$f(x)-bx-c=0$ has $i$ nonzero solutions in $\Fq$} \}|.
$$
Employing Theorem~\ref{thm-cubicmuldis}, we can derive the intersection distribution of some other monomials closely related to degree three polynomials.

\begin{theorem}\label{thm-intdis}
Let $f(x)=x^d$ be a polynomial over $\Fq$. Then the following holds.
\begin{itemize}
\item[(1)] If $p=2$, $m$ odd and $d \in \{ \frac{q+1}{3}, q-3 \}$, then
$$
v_0(f)=\frac{q^2-1}{3}, \quad v_1(f)=\frac{q^2-q+2}{2}, \quad v_2(f)=q-1, \quad v_3(f)=\frac{(q-1)(q-2)}{6}.
$$
\item[(2)] If $p=2$, $m$ even and $d=q-3$, then
$$
v_0(f)=\frac{(q-1)^2}{3}, \quad v_1(f)=\frac{3q^2+7q-4}{6}, \quad v_2(f)=0, \quad v_3(f)=\frac{(q-1)(q-4)}{6}, \quad v_4(f)=\frac{q-1}{3}.
$$
\item[(3)] If $p=3$ and $d \in \{\frac{2q}{3}, q-3\}$, then
$$
v_0(f)=\frac{(2q+3)(q-1)}{6}, \quad v_1(f)=\frac{q^2-2q+3}{2}, \quad v_2(f)=\frac{3(q-1)}{2}, \quad v_3(f)=\frac{(q-1)(q-3)}{6}.
$$
\item[(4)] If $p>3$, $q \equiv 1 \Mod{3}$ and $d=q-3$, then
\begin{align*}
&v_0(f)=\frac{(2q+1)(q-1)}{6}, \quad v_1(f)=\frac{3q^2-2q+5}{6}, \quad v_2(f)=\frac{3(q-1)}{2}, \\
&v_3(f)=\frac{(q-1)(q-7)}{6}, \quad v_4(f)=\frac{q-1}{3}.
\end{align*}
\item[(5)] If $p>3$, $q \equiv 2 \Mod{3}$ and $d \in \{ \frac{q+1}{3}, q-3\}$, then
$$
v_0(f)=\frac{(2q+5)(q-1)}{6}, \quad v_1(f)=\frac{q^2-4q+5}{2}, \quad v_2(f)=\frac{5(q-1)}{2}, \quad v_3(f)=\frac{(q-1)(q-5)}{6}.
$$
\end{itemize}
\end{theorem}
\begin{proof}
We only prove (4), since the other cases are similar. For $b,c \in\Fq$, consider the number of solutions to the equation $x^{q-3}-bx-c=0$. Note that $0$ is a solution if and only if $c=0$. Clearly,
\begin{equation*}
|\{x \in \Fq \mid x^{q-3}-bx=0 \}|=\begin{cases}
                                       1 & \mbox{if $b \notin C_0^{(3,q)}$,} \\
                                       4 & \mbox{if $b \in C_0^{(3,q)}$.}
\end{cases}
\end{equation*}
If $c \ne 0$, then it is easy to see that every nonzero solution to $x^{q-3}-bx-c=0$ is also a nonzero solution to $(\frac{1}{x})^3-\frac{c}{x}-b=0$. Hence, we need to count the number of nonzero solution to $x^3-cx-b=0$, where $b \in \Fq$ and $c \ne 0$. Note that $x^3-cx-b=0$ has a zero solution if and only if $b=0$. Moreover, $x^3-cx=0$ has $3$ solutions and $2$ nonzero solutions if and only if $c \in C_0^{(2,q)}$, and has $1$ solution and no nonzero solution if and only if $c \in C_1^{(2,q)}$. Thus, it remains to compute $M_i^*(x^3,c)$, for each $c \in \Fq^*$. Employing Theorem~\ref{thm-cubicmuldis}(3), we have
\begin{equation*}
\begin{cases}
M_0^*(x^3,c)=\frac{q-1}{3}, M_1^*(x^3,c)=\frac{q-j}{2}, M_2^*(x^3,c)=2+j, M_3^*(x^3,c)=\frac{q-10-3j}{6}, & \mbox{if $c \in C_0^{(2,q)}$,} \\
M_0^*(x^3,c)=\frac{q+2}{3}, M_1^*(x^3,c)=\frac{q+j-2}{2}, M_2^*(x^3,c)=1-j, M_3^*(x^3,c)=\frac{q-4+3j}{6}, & \mbox{if $c \in C_1^{(2,q)}$,}
\end{cases}
\end{equation*}
where
$$
j=\begin{cases}
     1 & \mbox{if $p \equiv 1 \Mod{12}$ or $m$ even,}\\
    -1 & \mbox{if $p \equiv 7 \Mod{12}$ and $m$ odd.}
\end{cases}
$$
Combining the above observations, we derive the intersection distribution.
\end{proof}

So far, not much is known about the non-hitting index of monomials. Employing Theorems~\ref{thm-cubicintdis} and \ref{thm-intdis}, in Table~\ref{tab-monononhitting}, we give an update of \cite[Table 3.2]{LP}, where an entry with superscript $\blacksquare$ represents the non-hitting index derived from Theorems~\ref{thm-cubicintdis} and \ref{thm-intdis}, an entry with superscript $\bigstar$ represents the non-hitting index which has not yet been understood, an entry without superscript represents the non-hitting index known before. Note that in the table, when $(d,q-1)=1$, we group $d$ and its inverse modulo $q-1$ together. As we shall see, when $q \le 11$, the non-hitting index of each monomial has been explained.

\begin{table}[H]
\begin{center}
\caption{The non-hitting index of all power mappings in $\Fq$, $q \le 16$}
\label{tab-monononhitting}
\begin{tabular}{|c|l|}
\hline
$q$ & $(d,v_0(x^d))$ \\ \hline
$2$  &  $(1,1)$   \\ \hline
$3$  &  $(1,2)$, $(2,3)$   \\ \hline
$4$  &  $(1,3)$, $(2,6)$, $(3,5)$   \\ \hline
$5$  &  $(1,4)$, $(2,10)$, $(3,8)$, $(4,7)$   \\ \hline
$7$  &  $(1,6)$, $(2,21)$, $(3,16)$, $(4,15)$, $(5,18)$, $(6,11)$ \\ \hline
$8$  &  $(1,7)$, $(\{2,4\},28)$, $(\{3,5\},21)$, $(6,28)$, $(7,13)$ \\ \hline
$9$  &  $(1,8)$, $(2,36)$, $(3,24)$, $(4,30)$, $(5,24)$, $(6,28)^\blacksquare$, $(7,32)$, $(8,15)$  \\ \hline
$11$  &  $(1,10)$, $(2,55)$, $(\{3,7\},40)^\blacksquare$, $(4,45)^\blacksquare$, $(5,38)$, $(6,35)$ , $(8,45)^\blacksquare$, $(9,50)$, $(10,19)$ \\ \hline
\multirow{2}{*}{$13$}  &  $(1,12)$, $(2,78)$, $(3,56)^\blacksquare$, $(4,57)^\bigstar$, $(5,60)^\bigstar$, $(6,58)$, $(7,48)$, $(8,69)^\bigstar$, $(9,56)^\bigstar$, \\
& $(10,54)^\blacksquare$, $(11,72)$, $(12,23)$ \\ \hline
\multirow{2}{*}{$16$}  &  $(1,15)$, $(\{2,8\},120)$, $(3,85)$, $(4,60)$, $(5,102)$, $(6,85)^\bigstar$, $(\{7,13\},75)^\blacksquare$, $(9,85)$, \\
& $(10,87)^\bigstar$, $(11,90)^\bigstar$, $(12,70)^\bigstar$, $(14,120)$, $(15,29)$ \\ \hline
\end{tabular}
\end{center}
\end{table}

\section{Monomials having the same intersection distribution as $x^3$}\label{sec3}

In this section, inspired by the open problem of classifying o-monomials, which is equivalent to finding all monomials over $\Fq=\F_{2^m}$ with the same intersection distribution as $x^2$, we consider monomials having the same intersection distribution as $x^3$. First of all, we display several classes of monomials satisfying this property.

\begin{theorem}\label{thm-sameintdiscubic}
\begin{itemize}
\item[(1)] When $p=2$ and $1 \le d \le q-1$, the monomial $x^d$ has intersection distribution \eqref{eqn-intdisdeg3othchar} in the following cases:
\begin{itemize}
\item[(1a)] $d=2^i+1$, $\gcd(i,m)=1$,
\item[(1b)] $d\equiv (2^i+1)^{-1} \Mod{q-1}$, $\gcd(i,m)=1$, $m$ odd,
\item[(1c)] $d\equiv -2^i \Mod{q-1}$, $\gcd(i,m)=1$, $m$ odd.
\end{itemize}
\item[(2)] When $p=3$, the monomial $x^d$ has intersection distribution \eqref{eqn-intdisdeg3char3} in the following case:
\begin{itemize}
\item[(2a)] $d=3^i$, $\gcd(i,m)=1$.
\end{itemize}
\item[(3)] When $p>3$, the monomial $x^d$ has intersection distribution \eqref{eqn-intdisdeg3othchar} in the following cases:
\begin{itemize}
\item[(3a)] $d=3$,
\item[(3b)] $d=\frac{2q-1}{3}$, $p \equiv 5 \Mod{6}$, $m$ odd, where $\frac{2q-1}{3}$ is the inverse of $3$ modulo $q-1$.
\end{itemize}
\end{itemize}
\end{theorem}

The proof of the above theorem follows from Remark~\ref{rem-monocubic} below. As we shall see, Theorem~\ref{thm-sameintdiscubic} contains the obvious monomials having the same intersection distribution as $x^3$. Besides, there are more monomials which are conjectured to have the same intersection distribution as $x^3$.

\begin{conjecture}\label{conj-sameintdiscubic}
A numerical experiment considers all monomials over $\Fq$ in the following ranges:
\begin{itemize}
\item[$\cdot$] $p=2$ and $1 \le m \le 21$,
\item[$\cdot$] $p=3$ and $1 \le m \le 13$,
\item[$\cdot$] $p>3$ and  $q \le 10^5$.
\end{itemize}
Based on the numerical result, we propose the following two conjectures.
\begin{itemize}
\item[(1)] The following two families of monomials $x^d$ over $\Fq=\F_{3^m}$ have intersection distribution \eqref{eqn-intdisdeg3char3}:
\begin{itemize}
\item[$\cdot$] $d=3^{(m+1)/2}+2$ and $d^{-1}$, $m$ odd,
\item[$\cdot$] $d=2 \cdot 3^{m-1}+1$ and $d^{-1}$, $m$ odd.
\end{itemize}
\item[(2)] The two families in Part (1), plus those in Theorem~\ref{thm-sameintdiscubic}, are all monomials having the same intersection distribution as $x^3$.
\end{itemize}
\end{conjecture}

\begin{remark}\label{rem-monocubic}
\quad
\begin{itemize}
\item[(1)] When $p=2$, Family (1a) in Theorem~\ref{thm-sameintdiscubic} contains quadratic monomials, whose intersection distribution follows from {\rm\cite[Theorem 5.6]{Blu}} {\rm (see also \cite[Proposition B.9]{LP})}. The monomials in (1a) are permutations whenever $m$ is odd. Their inverses are exactly those in Family (1b). The monomials in Family (1c) are closely related to quadratic monomials, since for each $b, c \in \Fq$, the equations $x^{-2^i}-bx-c=0$ and $bx^{2^i+1}-cx^{2^i}-1=0$ have the same nonzero solutions, and replace $x$ by $\frac{1}{y}$ in the latter one, we have $y^{2^i+1}-cy-b=0$, which goes back to the quadratic monomials case. Each monomial in Family (1c) has an inverse belonging to the same family.
\item[(2)] When $p>3$, the monomial in Family (3a) of Theorem~\ref{thm-sameintdiscubic} is a permutation if and only if $p \equiv 5 \Mod{6}$ and $m$ being odd. Hence, Family (3b) consists of the inverses of Family (3a) whenever they exist.
\item[(3)] According to Parts (1) and (2), when $p \ne 3$, all monomials having the same intersection distribution as $x^3$, are the obvious ones. In contrast, the $p=3$ case is more interesting since some less obvious monomials occur. On one hand, the Family (2a) contains linearized monomials, whose proof is easy {\rm(see for instance \cite[Table 3.1]{LP})}. Moreover, each monomial in Family (2a) has an inverse belonging to the same family. On the other hand, the two more families in Conjecture~\ref{conj-sameintdiscubic}(1) are still mysterious.
\item[(4)] It is worthy to note that the exponents in Conjecture~\ref{conj-sameintdiscubic}(1) are all three-valued decimations in regard to the cross-correlation distribution of ternary $m$-sequences {\rm(see \cite[Theorem 6(A)]{DHKM} and \cite[Theorem 4.9]{Hel})}. We note that for $1 \le i \le m-1$, the decimations $d$ and $3^id$ have the same cross-correlation distribution. On the other hand, we think the intersection distribution is a more subtle property, since for $1 \le i \le m-1$, $x^d$ and $x^{3^id}$ over $\F_{3^m}$ have different intersection distributions in general.
\end{itemize}
\end{remark}

Next, we make some progress towards Conjecture~\ref{conj-sameintdiscubic}(2), by providing some restrictions on the monomials satisfying \eqref{eqn-intdisdeg3othchar} or \eqref{eqn-intdisdeg3char3}. Recall that an affine line is an $i$-secant line to $G_f$, if it intersects $G_f$ in exactly $i$ points. Since every pair of distinct points in $G_f$ could determine a $2$-secant line, the largest value of $v_2(f)$ is $\frac{q(q-1)}{2}$. In this sense, we observe that $v_2(f)=q-1$ in \eqref{eqn-intdisdeg3othchar} and $v_2(f)=0$ in \eqref{eqn-intdisdeg3char3} are both very small, which means there are very few $2$-secant lines to $G_f$. Next, we are going to show that this unusual geometric property can be interpreted in an algebraic way, which gives strong restrictions on the monomials satisfying \eqref{eqn-intdisdeg3othchar} or \eqref{eqn-intdisdeg3char3}.

Considering monomials with intersection distribution \eqref{eqn-intdisdeg3othchar}, we need to understand under what conditions, there are exactly $q-1$ distinct $2$-secant lines to $G_f$. As a preparation, we have the following two lemmas. We write $g_d(x)=\frac{x^d-1}{x-1}$ and use $H_{q,d}=\{ g_d(x) \mid x \in \Fq \sm \{1\}\}$ to denote the image set of $g_d(x)$ over $\Fq \sm \{1\}$. The following lemma is easy to see.

\begin{lemma}\label{lem-secpreima}
Let $f$ be a polynomial over $\Fq$. For distinct $x_1, x_2 \in \Fq$ with $x_1 \ne 0$, write $y=\frac{x_2}{x_1} \in \Fq \sm \{1\}$. Then we have the following.
\begin{itemize}
\item[(1)] Two points $(x_1,f(x_1)), (x_2,f(x_2)) \in G_f$ determine a $2$-secant line to $G_f$ if and only if the equation $\frac{f(x)-f(x_1)}{x-x_1}=\frac{f(x_2)-f(x_1)}{x_2-x_1}$ has exactly one solution $x=x_2$. In particular, if $f(x)=x^d$, then the two points $(x_1,x_1^d), (x_2,x_2^d) \in G_f$ determine a $2$-secant line to $G_f$ if and only if $\frac{y^d-1}{y-1} \in H_{q,d}$ has exactly one preimage $y$ under $g_d$. Furthermore, if $x_2 \ne 0$ or equivalently $y \ne 0$, by interchanging the roles of $x_1$ and $x_2$, the two points $(x_1,x_1^d), (x_2,x_2^d) \in G_f$ determine a $2$-secant line to $G_f$ if and only if $\frac{(1/y)^d-1}{1/y-1} \in H_{q,d}$ has exactly one preimage $1/y$ under $g_d$.
\item[(2)] For $f(x)=x^d$ and each $y \in \Fq \sm \{1\}$ such that $\frac{y^d-1}{y-1}$ has exactly one preimage $y$ under $g_d$, the $q-1$ pairs of distinct points $\{\{ (x_1,x_1^d), (x_2,x_2^d) \} \mid x_1, x_2 \in \Fq^*, \frac{x_2}{x_1}=y \}$ determine $q-1$ distinct $2$-secant lines to $G_f$. Moreover, suppose
    $$
    \{ y \in \Fq \sm\{1\} \mid \mbox{$\frac{y^d-1}{y-1}$ has exactly one preimage $y$ under $g_d$} \}=\{ y_1,y_1^{-1},y_2,y_2^{-1},\cdots,y_s,y_s^{-1},y_1^{\pr},y_2^{\pr},\cdots,y_t^{\pr} \},
    $$
    where no element in $\{y_1^{\pr},y_2^{\pr},\cdots,y_t^{\pr} \}$ is the inverse of any other element. Then there are exactly $(s+t)(q-1)$ distinct $2$-secant lines to $G_f$.
\item[(3)] The two points $(x_1,f(x_1)), (x_2,f(x_2)) \in G_f$ determine a $3$-secant line to $G_f$ if and only if the equation $\frac{f(x)-f(x_1)}{x-x_1}=\frac{f(x_2)-f(x_1)}{x_2-x_1}$ has exactly two solutions. In particular, if $f(x)=x^d$, then the two points $(x_1,x_1^d), (x_2,x_2^d) \in G_f$ determine a $3$-secant line to $G_f$ if and only if $\frac{y^d-1}{y-1} \in H_{q,d}$ has exactly two preimages under $g_d$.
\end{itemize}
\end{lemma}

In the case that $z \in H_{q,d}$ has exactly one preimage under $g_d$, we have the following lemma providing crucial information about the images and preimages of $g_d$.

\begin{lemma}\label{lem-gdpreima}
Let $f(x)=x^d$ be over $\Fq$. Suppose $z \in H_{q,d}$ has exactly one preimage $y \in \Fq \sm \{1\}$ under $g_d$. Then we have the following:
\begin{itemize}
\item[(1)] If $z=0$, then $q$ is odd, $d$ is even and $y=-1$.
\item[(2)] If $y \notin \{0,-1\}$, then $z \notin \{0,1\}$, and $y^{-d+1}z \notin \{0,1,z\}$ also has exactly one preimage $\frac{1}{y} \in \Fq \sm \{0,\pm1\}$.
\item[(3)] If $q$ is even and $H_{q,d}$ has exactly one element $z$ with exactly one preiamge $y$ under $g_d$, then $(y,z)=(0,1)$.
\item[(4)] If $q$ is odd and $H_{q,d}$ has exactly two elements $z, z^{\pr}$ with exactly one preimage under $g_d$, say $y, y^{\pr}$ respectively, then either $(y,z)=(0,1)$ and $(y^{\pr},z^{\pr})=(-1,0)$, or $y \notin \{0,-1\}$, $y^{\pr}=\frac{1}{y}$ and $z^{\pr}=y^{-d+1}z$.
\end{itemize}
\end{lemma}
\begin{proof}
(1) If $0 \in H_{q,d}$ has exactly one preimage $y$ under $g_d$, then $(d,q-1)=2$. Thus $q$ is odd and $d$ is even, which implies $y=-1$.

(2) Since $y \ne -1$, by Part (1), $z \ne 0$ and $\frac{1}{y} \ne -1$. Since $y \ne 0$, then $z \ne 1$ and $y^{-d+1} \ne 1$, which implies $y^{-d+1}z \ne z$ and $y \ne 1$. Since $z=\frac{y^d-1}{y-1}$, we have $y^{-d+1}z= \frac{(1/y)^d-1}{(1/y)-1}$. By Lemma~\ref{lem-secpreima}(1), $y^{-d+1}z= \frac{(1/y)^d-1}{(1/y)-1}$ has exactly one preimage $\frac{1}{y} \in \Fq \sm \{0, \pm 1\}$ under $g_d$. As the image of $\frac{1}{y} \notin \{0,\pm 1\}$ under $g_d$, the element $y^{-d+1}z \notin \{0,1\}$.

(3) Since $H_{q,d}$ has exactly one element $z$ with exactly one preimage $y$ under $g_d$, by Part (2), $y \in \{0,-1\}$. Note that $q$ being even forces $y=0$. Consequently, $(y,z)=(0,1)$.

(4) If $y \notin \{0,-1\}$, then by Part (2), we have $y^{\pr}=\frac{1}{y}$ and $z^{\pr}=y^{-d+1}z$. If $y \in \{0,-1\}$, first, assume $y=0$ and therefore $z=1$. Suppose $y^{\pr} \notin \{0,-1\}$, then by Part (2), $z=1,z^{\pr},y^{\pr -d+1}z^{\pr}$ are distinct and all have exactly one preimage under $g_d$, which is impossible. Hence, $y^{\pr}=-1$ and $z^{\pr} \in \{0,1\}$. Note that $z=1$ has exactly one preimage, then $z^{\pr} \ne 1$ and $(y^{\pr},z^{\pr})=(-1,0)$.
\end{proof}

Now we are ready to derive some restrictions on monomials satisfying \eqref{eqn-intdisdeg3othchar}.

\begin{theorem}\label{thm-necessary}
Let $f(x)=x^d$ be over $\Fq$ satisfying \eqref{eqn-intdisdeg3othchar}. Then
\begin{itemize}
\item[(1)] Each element in $H_{q,d}$ has either one or two preimages under $g_d$. Furthermore, the number of elements in $H_{q,d}$ having exactly one preimage under $g_d$ is either one or two.
\item[(2)] If $q$ is even, then there exists exactly one element $z \in H_{q,d}$ with exactly one preimage $y$ under $g_d$, where $(y,z)=(0,1)$.
\item[(3)] If $q$ is odd, then there exist exactly two elements $z, z^{\pr} \in H_{q,d}$ with exactly one preimage under $g_d$, say $y, y^{\pr}$ respectively, where $y \notin \{0,-1\}$, $y^{\pr}=\frac{1}{y}$ and $z^{\pr}=y^{-d+1}z$.
\item[(4)] $$
             (d,q-1)=\begin{cases}
               1 & \mbox{if $0 \notin H_{q,d}$,} \\
               3 & \mbox{if $0 \in H_{q,d}$.}
             \end{cases}
           $$
\end{itemize}
\end{theorem}
\begin{proof}
(1) Since $v_i(f)=0$ for each $i>3$, then by Lemma~\ref{lem-secpreima}, each element in $H_{q,d}$ has either one or two preimages. Consider the number of elements in $H_{q,d}$, which has exactly one preimage. By Lemmas~\ref{lem-secpreima}(2), if this number is either zero or at least three, then $v_2(f)=0$ or $v_2(f) \ge 2(q-1)$, which contradicts \eqref{eqn-intdisdeg3othchar}. Hence, the number is either one or two.

(2) If $q$ is even, then the preimage set $\Fq \sm \{1\}$ has odd size $q-1$. Combining Part (1) and the parity, there exists exactly one element $z$ in $H_{q,d}$, which has exactly one preimage $y$ under $g_d$. By Lemma~\ref{lem-gdpreima}(3), we have $(y,z)=(0,1)$.

(3) If $q$ is odd, then the preimage set $\Fq \sm \{1\}$ has even size $q-1$. Combining Part (1) and the parity, there exists two elements $z, z^{\pr} \in H_{q,d}$ with exactly one preimage under $g_d$. Suppose $z=g_d(y)$ and $z^{\pr}=g_d(y^{\pr})$. By Lemma~\ref{lem-gdpreima}(4), we have either $(y,z)=(0,1)$ and $(y^{\pr},z^{\pr})=(-1,0)$, or $y \notin \{0,-1\}$, $y^{\pr}=\frac{1}{y}$ and $z^{\pr}=y^{-d+1}z$. According to Lemma~\ref{lem-secpreima}(2), the former case implies $v_2(f) \ge 2(q-1)$, which contradicts \eqref{eqn-intdisdeg3othchar}.

(4) If $0 \notin H_{q,d}$, then clearly $(d,q-1)=1$. If $0 \in H_{q,d}$, then $0$ has exactly either one or two preimages under $g_d$. If $0$ has exactly one preimage under $g_d$, then by Lemma~\ref{lem-gdpreima}(1), we have $q$ being odd and the preimage is $-1$. This contradicts Part (3). Hence, $0$ has exactly two preimages under $g_d$ and therefore $(d,q-1)=3$.
\end{proof}

Consequently, we have the following necessary and sufficient condition characterizing monomials over $\Fq$ satisfying \eqref{eqn-intdisdeg3othchar}, when $q$ is not divisible by $3$.

\begin{theorem}\label{thm-necesuffcubic}
Let $f(x)=x^d$ be over $\Fq$ where $(3,q)=1$. Then $f$ satisfies \eqref{eqn-intdisdeg3othchar} if and only if
one of the following holds.
\begin{itemize}
\item[(1)] If $q$ is even, then $0$ is the only preimage of $1 \in H_{q,d}$ under $g_d$ and $g_d|_{\Fq \sm \{0,1\}}$ is $2$-to-$1$.
\item[(2)] If $q$ is odd, then there exist exactly two elements $z, z^{\pr} \in H_{q,d}$ with exactly one preimage under $g_d$, say $y, y^{\pr}$ respectively, where $y \notin \{0,-1\}$, $y^{\pr}=\frac{1}{y}$ and $z^{\pr}=y^{-d+1}z$, and $g_d|_{\Fq\sm \{1,y,y^{\pr}\}}$ is $2$-to-$1$.
\end{itemize}
In both $q$ even and odd cases, we have
$$
(d,q-1)=\begin{cases}
           1 & \mbox{if $0 \notin H_{q,d}$,} \\
           3 & \mbox{if $0 \in H_{q,d}$.}
        \end{cases}
$$
\end{theorem}
\begin{proof}
The necessity follows from Theorem~\ref{thm-necessary} and we only need to consider the sufficiency. For Part (1), by employing Lemma~\ref{lem-secpreima} and Theorem~\ref{thm-necessary}(2), we have $v_2(f)=q-1$ and $v_i(f)=0$ for each $i>3$. For Part (2), by employing Lemma~\ref{lem-secpreima} and Theorem~\ref{thm-necessary}(3), we have $v_2(f)=q-1$ and $v_i(f)=0$ for each $i>3$. Together with Proposition~\ref{prop-basiceqn}, we conclude that $f$ satisfies \eqref{eqn-intdisdeg3othchar}. The greatest common divisor $(d,q-1)$ follows from the $2$-to-$1$ property.
\end{proof}

Similarly, we have the following necessary and sufficient condition characterizing monomials over $\Fq$ satisfying \eqref{eqn-intdisdeg3char3}, when $q$ is a power of $3$.

\begin{theorem}\label{thm-necesuffcubicchar3}
Let $f(x)$ be over $\Fq=\F_{3^m}$. Then $f$ satisfies \eqref{eqn-intdisdeg3char3} if and only if for each $y \in \Fq$, the function $\left.\frac{f(x+y)-f(y)}{x} \right|_{\Fq^*}$ is $2$-to-$1$. In particular, $f(x)=x^d$ satisfies \eqref{eqn-intdisdeg3char3} if and only if the following holds:
\begin{itemize}
\item[(1)] $\gcd(d-1,q-1)=2$,
\item[(2)] $\left.g_d\right|_{\Fq \sm \{1\}}$ is $2$-to-$1$, which implies
           $$
             (d,q-1)=\begin{cases}
               1 & \mbox{if $0 \notin H_{q,d}$,} \\
               3 & \mbox{if $0 \in H_{q,d}$.}
             \end{cases}
           $$
\end{itemize}
\end{theorem}
\begin{proof}
A polynomial $f$ has intersection distribution \eqref{eqn-intdisdeg3char3} if and only if every two distinct points in $G_f$ lead to a unique third point in $G_f$, which lies on the line determined by these two points. Hence, for two distinct $x_1, x_2 \in \Fq$, the equation $\frac{f(x)-f(x_1)}{x-x_1}=\frac{f(x_2)-f(x_1)}{x_2-x_1}$ has a unique solution $x \in \Fq \sm \{x_1,x_2\}$. Equivalently, for each $y \in \Fq$, we have $\left.\frac{f(x)-f(y)}{x-y} \right|_{\Fq \sm \{y\}}$ is $2$-to-$1$. Therefore, $f$ has intersection distribution \eqref{eqn-intdisdeg3char3} if and only if the function $\left.\frac{f(x+y)-f(y)}{x} \right|_{\Fq^*}$ is $2$-to-$1$ for each $y \in \Fq$. Consider $f(x)=x^d$. For $y=0$, the mapping $\left.\frac{f(x)}{x}=x^{d-1} \right|_{\Fq^*}$ is $2$-to-$1$ if and only if $(d-1,q-1)=2$. For $y \in \Fq^*$, the mapping $\left.\frac{f(x)-f(y)}{x-y}=\frac{x^d-y^d}{x-y} \right|_{\Fq \sm \{y\}}$ is $2$-to-$1$ if and only if $\left.g_d\right|_{\Fq \sm \{1\}}$ is $2$-to-$1$. The greatest common divisor $(d,q-1)$ follows from the $2$-to-$1$ property.
\end{proof}

Theorems~\ref{thm-necesuffcubic} and \ref{thm-necesuffcubicchar3} can be viewed as analogies of \cite[Theorem 8.22, Corollary 8.24]{Hirs}, which give characterizations of o-polynomials and o-monomials. The strict restrictions in these theorems indicate that monomials with intersection distribution \eqref{eqn-intdisdeg3othchar} or \eqref{eqn-intdisdeg3char3} are very rare. Actually, these two theorems help to significantly reduce the computational complexity of verifying whether a monomial has intersection distribution \eqref{eqn-intdisdeg3othchar} or \eqref{eqn-intdisdeg3char3}, which leads to Conjecture~\ref{conj-sameintdiscubic}.

\section{Nonisomorphic Steiner triple systems arising from monomials}\label{sec4}

In this section, we shall observe that a polynomial over $\F_{3^m}$ with intersection distribution \eqref{eqn-intdisdeg3char3} produces a Steiner triple system. More interestingly, some nonisomorphic Steiner triple systems are obtained by employing distinct polynomials over $\F_{3^m}$.

Recall that a Steiner triple system $\STS(v)$ is a set system $(\cV,\cB)$, where $\cV$ is a point set of $v$ elements, and $\cB$ is a block set consisting of distinct $3$-subsets, such that every two points are contained in exactly one block. Two Steiner triple systems $(\cV_1,\cB_1)$ and $(\cV_2,\cB_2)$ are isomorphic, if there exists a bijection between $\cV_1$ and $\cV_2$, which also induces a bijection between $\cB_1$ and $\cB_2$. For a comprehensive survey about Steiner triple systems, please refer to \cite[Section II.2]{Col}. The following is a primary example of Steiner triple systems $\STS(v)$ with $v$ being a power of $3$.

\begin{example}{\rm (\cite[Section II.2, Theorem 2.10]{Col})}
Let $\cV=\F_{3^m}$ and
$$
\cB=\{\{x_1,x_2,x_3\} \mid \mbox{$x_1,x_2,x_3 \in \F_{3^m}$ distinct, $x_1+x_2+x_3=0$}\}.
$$
Then $(\cV,\cB)$ forms an $\STS(3^m)$. In another word, the points and lines in the affine geometry $AG(m,3)$ generate an $\STS(3^m)$, which is therefore named an \emph{affine triple system}.
\end{example}

Next, we propose a construction of $\STS(3^m)$ arising from polynomials over $\F_{3^m}$. For a polynomial $f$ over $\Fq$, it is called \emph{$\Fp$-linearized}, if $f(x+y)=f(x)+f(y)$ for each $x,y \in \Fq$ and $f(ax)=af(x)$ for each $x \in \Fq$ and $a \in \Fp$. Note that $f$ is $\Fp$-linearized only if $f(0)=0$.

\begin{theorem}\label{thm-STS}
Let $f$ be a polynomial over $\F_{3^m}$ intersection distribution \eqref{eqn-intdisdeg3char3}. Let $\cV=\F_{3^m}$ and
$$
\cB_f=\{ \{x_1,x_2,x_3\} \mid \mbox{$x_1,x_2,x_3 \in \F_{3^m}$ distinct, $\frac{f(x_3)-f(x_1)}{x_3-x_1}=\frac{f(x_2)-f(x_1)}{x_2-x_1}$}\}.
$$
Then $(\cV,\cB_f)$ is an $\STS(3^m)$. Moreover,
\begin{itemize}
\item[(1)] if $f^{\pr}(x)=f(x)+bx+c$, then $(\cV,\cB_f)$ and $(\cV,\cB_{f^{\pr}})$ are the same.
\item[(2)] if $f$ is a permutation, then $(\cV,\cB_f)$ and $(\cV,\cB_{f^{-1}})$ are isomorphic.
\item[(3)] if $f(0)=0$, then $(\cV,\cB_f)$ is an affine triple system if and only if $f$ is an $\F_3$-linearized polynomial.
\end{itemize}
\end{theorem}
\begin{proof}
Note that $f$ has intersection distribution \eqref{eqn-intdisdeg3char3}. By Theorem~\ref{thm-necesuffcubicchar3}, for each pair of distinct elements $x_1, x_2 \in \F_{3^m}$, there is a unique $x_3 \in \F_{3^m}$ different from $x_1$ and $x_2$, such that $\frac{f(x_3)-f(x_1)}{x_3-x_1}=\frac{f(x_2)-f(x_1)}{x_2-x_1}$. Hence, $\cB_f$ is well-defined. Since every pair of distinct elements $x_1$ and $x_2$ determines a unique $x_3$, which form a block $\{x_1, x_2, x_3\}$, then $(\cV,\cB_f)$ is an $\STS(3^m)$ by definition.

Part (1) follows easily from the definition of $\cB_f$ and $\cB_{f^{\pr}}$. For Part (2), note that $\{ x_1,x_2,x_3 \} \in \cB_{f}$ if and only if $\{ f(x_1),f(x_2),f(x_3) \} \in \cB_{f^{-1}}$. Therefore, $f$ is a bijection of $\cV$ and induces a bijection between $\cB_{f}$ and $\cB_{f^{-1}}$. Thus, $(\cV,\cB_f)$ and $(\cV,\cB_{f^{-1}})$ are isomorphic. For Part (3), we can see that assuming $f(0)=0$ does not lose any generality by Part (1). If $f$ is $\F_3$-linearized, then $x_3=-x_1-x_2$ is the unique solution to $\frac{f(x)-f(x_1)}{x-x_1}=\frac{f(x_2)-f(x_1)}{x_2-x_1}$, which leads to an affine triple system. Conversely, if $(\cV,\cB_f)$ is an affine triple system, then the summation of the elements in each block is $0$. Hence, for each pair of distinct elements $x_1$ and $x_2$, we have $x_3=-x_1-x_2$ and $\frac{f(-x_1-x_2)-f(x_1)}{x_1-x_2}=\frac{f(x_2)-f(x_1)}{x_2-x_1}$, which implies $f(-x_1-x_2)=-f(x_1)-f(x_2)$. Set $x_2=0$, we have $f(-x_1)=-f(x_1)$ and therefore, $f(-x_1-x_2)=f(-x_1)+f(-x_2)$. Hence, $f$ is an $\F_3$-linearized polynomial over $\F_{3^m}$.
\end{proof}

Combining Theorems~\ref{thm-sameintdiscubic}(2) and \ref{thm-STS}, the affine triple system $\STS(3^m)$ can be derived by using monomial $f(x)=x^{3^i}$ over $\F_{3^m}$, where $(i,m)=1$. We ask if there are other polynomials over $\F_{3^m}$, which produces Steiner triple system nonisomorphic to the affine ones. In view of Conjecture~\ref{conj-sameintdiscubic}(1), we compare the Steiner triple systems derived from the two families in it and the affine triple systems when $m$ is small.

\begin{example}\label{exm-Steiner}
For $m$ being odd, let $f_1(x)=x^3$, $f_2(x)=x^{3^{(m+1)/2}+2}$ and $f_3(x)=x^{2\cdot3^{m-1}+1}$ be polynomials over $\F_{3^m}$. According to Conjecture~\ref{conj-sameintdiscubic}, $f_1$, $f_2$ and $f_3$ have the same intersection distribution for $1 \le m \le 13$, $m$ odd. Let $\cV=\F_{3^m}$. A numerical experiment indicates the following.
\begin{itemize}
\item[(1)] When $m=3$, since the two permutations $f_2(x)=x^{11}$ and $f_3(x)=x^{19}$ are inverses of each other, then by Theorem~\ref{thm-STS}(1), $(\cV,\cB_{f_2})$ and $(\cV,\cB_{f_3})$ are nonisomorphic. Moreover, $(\cV,\cB_{f_1})$ and $(\cV,\cB_{f_2})$ are nonisomorphic.
\item[(2)] When $m=5$, $(\cV,\cB_{f_1})$, $(\cV,\cB_{f_2})$ and $(\cV,\cB_{f_3})$ are pairwise nonisomorphic.
\end{itemize}
Assuming that Conjecture~\ref{conj-sameintdiscubic}(1) is true, we predict that $f_i$, $1 \le i \le 3$, produce three pairwise nonisomorphic $\STS(3^m)$ when $m \ge 5$.
\end{example}

\section{Application to Kakeya sets in affine planes}\label{sec5}

Let $\ell$ be the line at infinity in $PG(2,q)$. For each point $P \in \ell$, define $\ell_P$ to be a line through $P$ other than $\ell$. A \emph{Kakeya set} in $PG(2,q)$ is defined to be the point set
$$
K=(\bigcup_{P \in \ell} \ell_P) \sm \ell.
$$
If we restrict to the affine plane $AG(2,q)=PG(2,q) \sm \ell$, then the Kakeya set $K$ contains an affine line in each direction. So far, most papers concerning Kakeya sets in affine planes focus on Kakeya sets whose sizes are close to the lower and upper bounds \cite{BDMS,BM,BV,DM1,DM2,DMS,Fab}. Note that the construction of Kakeya sets is easy, since for each point $P \in \ell$, we can choose an arbitrary line $\ell_P$ through $P$ other than $\ell$. On the other hand, computing the size of Kakeya set is difficult. In \cite{DM2}, an exhaustive search determines all possible sizes of Kakeya sets in $PG(2,q)$ where $q \le 9$. Inspired by this work, the authors of \cite{LP} proposed explicitly constructions of Kakeya sets with nice underlying algebraic structures, which are derived from monomials over finite fields and have previously unknown sizes. Along this line, we present infinite families of Kakeya sets from degree three polynomials in this section. As a major advantage of our constructions, the sizes of proposed Kakeya sets follow directly from the multiplicity distribution of degree three polynomials, which have been computed in Section~\ref{sec2}. For Kakeya sets in affine spaces with higher dimension, please refer to \cite{Dvir,KLSS,KMW,Mas,Wol}.

First of all, we remark that the concept of intersection distribution can be defined with respect to point sets in classical projective planes $PG(2,q)$ \cite[Definition 1.3]{LP}.

\begin{definition}\label{def-set}
Let $D$ be a point set in $PG(2,q)$. For $0 \le i \le q+1$, define $u_i(D)$ to be the number of lines in $PG(2,q)$, which intersect $D$ in exactly $i$ points. The sequence $(u_i(D))_{i=0}^{q+1}$ is the intersection distribution of $D$. The integer $u_0(D)$ is the non-hitting index of $D$.
\end{definition}

For a $(q+2)$-set $D$ in $PG(2,q)$, a point $P \in D$ is called an \emph{internal nucleus} of $D$, if each line through $P$ intersects $D$ in exactly one more point. The following is an alternative viewpoint to understand Kakeya sets proposed in \cite{BM}.

\begin{lemma}{\rm \cite[Lemma 4.1]{LP}}\label{lem-DK}
Let $K$ be a Kakeya set in $PG(2,q)$, where $K=(\bigcup_{P \in \ell} \ell_P) \sm \ell$. Define the dual Kakeya set $DK$ to be the dual of the $q+2$ lines $\{\ell_P \mid P \in \ell\} \cup \{\ell\}$. Then $DK$ is a $(q+2)$-set in $PG(2,q)$ with an internal nucleus, such that $|K|=q^2-u_0(DK)$.
\end{lemma}

Therefore, computing the size of $K$ amounts to calculating the non-hitting index of the dual Kakeya set $DK$. Moreover, to construct a Kakeya set, it suffices to construct its dual, which is a $(q+2)$-set in $PG(2,q)$ with an internal nucleus. Actually, given a polynomial $f$ over $\Fq$ and $b \in \Fq$, we construct a dual Kakeya set
$$
DK(f,b):=\{ \lan (x,f(x),1) \ran \mid x \in \Fq \} \cup \{ \lan (0,1,0),(1,b,0)\ran \},
$$
which has an internal nucleus $\lan (0,1,0) \ran$. Indeed, the non-hitting index of $DK(f,b)$ follows from the multiplicity distribution of $f$ \cite[Proposition 4.3]{LP}. Consequently, we have the following proposition. Note that for a dual Kakeya set $DK(f,b)$, we use $K(f,b)$ to denote the Kakeya set dual to $DK(f,b)$.

\begin{proposition}\label{prop-intdisKconnection}
For a polynomial $f$ over $\Fq$ and $b \in \Fq$, we have $u_0(DK(f,b))=v_0(f)-M_0(f,b)$ and therefore, $|K(f,b)|=q^2-v_0(f)+M_0(f,b)$. 
\end{proposition}

Therefore, the non-hitting index $v_0(f)$ and the intersection distribution $M_0(f,b)$ imply the size of $K(f,b)$. Combining Theorems~\ref{thm-cubicintdis}, \ref{thm-cubicmuldis} and Proposition~\ref{prop-intdisKconnection}, we can obtain the size of some Kakeya sets derived from monomials.

\begin{theorem}\label{thm-cubicKakeya}
\quad
\begin{itemize}
\item[(1)] Suppose $q \equiv 0 \Mod{3}$. Then $|K(x^3-ax^2,b)|=\frac{2q^2+q}{3}$ if $a=0$, $b \notin C_0^{(2,q)}$, or $a \ne 0$, $b \in \Fq$, and $|K(x^3,b)|=\frac{2q^2+3q}{3}$ if $b \in C_0^{(2,q)}$.
\item[(2)] Suppose $q \equiv 1 \Mod{3}$. Then $|K(x^3-ax^2,b)|=\frac{2q^2+2q-1}{3}$ if $a=b=0$, or $a \ne 0$, $\frac{b}{a^2}=-\frac{1}{3}$, and $|K(x^3-ax^2,b)|=\frac{2q^2+q}{3}$ if $a=0$, $b \ne 0$, or $a \ne 0$, $\frac{b}{a^2}\ne-\frac{1}{3}$.
\item[(3)] Suppose $q \equiv 2 \Mod{3}$. Then $|K(x^3-ax^2,b)|=\frac{2q^2+1}{3}$ if $a=b=0$, or $a \ne 0$, $\frac{b}{a^2}=-\frac{1}{3}$, and $|K(x^3-ax^2,b)|=\frac{2q^2+q+2}{3}$ if $a=0$, $b \ne 0$, or $a \ne 0$, $\frac{b}{a^2} \ne -\frac{1}{3}$.
\end{itemize}
\end{theorem}

In Table~\ref{tab-Kakeya}, we list the known sizes of Kakeya sets in $PG(2,q)$, with prime power $q \le 19$. When $q \le 9$, all possible sizes follow from the exhaustive search in \cite[Table 1]{DM2}. When $11 \le q \le 19$, we only list the known sizes realizable by explicit constructions. We note that when the sizes of the Kakeya sets are close to the lower and upper bounds, there have been a series of literature concerning their construction and characterization \cite{BDMS,BM,BV,DM1,DM2,DMS,Fab}. In the table, an entry with superscript $\blacksquare$ represents the size of Kakeya sets following from the explicit constructions in Theorem~\ref{thm-cubicKakeya}, which are unknown before. An entry with superscript $\bigstar$ represents the size of Kakeya sets which do not have explicit constructions so far. An entry without superscript represents the size of Kakeya sets with known explicit constructions before.

\begin{table}[H]
\ra{1.3}
\caption{The known sizes of Kekaya set in $PG(2,q)$, for prime power $q \le 19$}
\label{tab-Kakeya}
\begin{center}
\begin{tabular}{|c|c|}
\hline
$q$ & Sizes of Kekaya sets \\ \hline
$2$ & $3$, $4$ \\ \hline
$3$ & $7$, $9$ \\ \hline
$4$ & $10$, $12$, $13$, $16$ \\ \hline
$5$ & $17$, $18$, $19$, $21$, $25$ \\ \hline
$7$ & $31$, $32^\bigstar$, $33$, $34$, $35$, $36^\bigstar$, $37$, $39$, $43$, $49$ \\ \hline
$8$ & $36$, $40$, $42$, $43$, $44^\bigstar$, $45^\bigstar$, $46$, $47^\bigstar$, $48$, $49$, $52$, $57$, $64$ \\ \hline
$9$ & $49$, $51^\bigstar$, $52$, $53$, $54$, $55$, $56^\bigstar$, $57$, $58^\bigstar$, $59^\bigstar$, $60^\bigstar$, $61$, $62^\bigstar$, $63$, $67$, $73$, $81$ \\ \hline
$11$ & $71$, $75$, $77$, $81^\blacksquare$, $85$, $86$, $87$, $91$, $93$, $97$, $103$, $111$, $121$ \\ \hline
$13$ & $97$, $103$, $112$, $115$, $117^\blacksquare$, $121$, $127$, $129$, $133$, $139$, $147$, $157$, $169$ \\ \hline
$16$ & $136$, $144$, $148$, $150$, $160$, $166$, $176$, $181$, $192$, $193$, $196$, $201$, $208$, $217$, $228$, $241$, $256$ \\ \hline
$17$ & $161$, $169$, $189$, $193^\blacksquare$, $199^\blacksquare$, $200$, $209$, $217$, $219$, $223$, $229$, $237$, $247$, $259$, $273$, $289$ \\ \hline
$19$ & $199$, $207$, $209$, $247^\blacksquare$, $253^\blacksquare$, $259$, $261$, $262$, $271$, $273$, $277$, $283$, $291$, $301$, $313$, $327$, $343$, $361$\\ \hline
\end{tabular}
\end{center}
\end{table}

\section{Conclusion}\label{sec6}

In this paper, we determined the multiplicity distribution of polynomials with the form $x^3-ax^2$, which gives the intersection distribution of each degree three polynomial. Inspired by the famous open problem of classifying o-polynomials, we initiated to classify all monomials having the same intersection distribution as $x^3$ and made some progress along this line. Interestingly, when $p=3$, numerical experiment indicated that some monomials with the same intersection distribution as $x^3$ led to nonisomorphic Steiner triple systems. Finally, The multiplicity distribution of $x^3-ax^2$ generated several families of Kakeya sets, whose sizes are different comparing with the known ones.

Except Conjecture~\ref{conj-sameintdiscubic}, we think the following three problems deserve further investigation.

\begin{itemize}
\item[(1)] In Table~\ref{tab-monononhitting}, the non-hitting indices of certain monomials have not been well understood. Therefore, it is interesting to give a theoretical explanation for these non-hitting indices.

\item[(2)] In Example~\ref{exm-Steiner}, the fact that the Steiner triple systems being nonisomorphic follows from a numerical computation. A theoretic proof confirming the nonisomorphism, even only for small values of $m$, could be very enlightening.

\item[(3)] In Table~\ref{tab-Kakeya}, there are a few Kakeya sets having no theoretical constructions, whose sizes are only known by numerical experiment. We ask for explicit constructions for these Kakeya sets.
\end{itemize}

We finally mention a recent work due to Ding and Tang \cite{DT}, in which polynomials over finite fields were employed to construct combinatorial $t$-designs. While determining the parameters of the $t$-design arising from a polynomial $f$ is difficult in general \cite{DT}, we note that the multiplicity distribution of $f$ implies the parameters of the associated $t$-design. Therefore, this design-theoretic application supplies one more motivation to study the multiplicity distribution of polynomials over finite fields.

\section*{Acknowledgement}

Shuxing Li is supported by the Alexander von Humboldt Foundation. This project was initiated during the second and third authors visited the first author at University of Rostock. They wish to thank the Institute of Mathematics for the great hospitality. The second author wishes to thank Christian Kaspers, Otto von Guerick University of Magdeburg, for his kind help with the numerical experiments in Conjecture~\ref{conj-sameintdiscubic}.

\end{document}